\documentclass[a4paper,oneside,10pt]{article}%
\usepackage{amsmath}
\usepackage{amsfonts}
\usepackage{amssymb}
\usepackage{graphicx}
\usepackage[square,numbers,sort&compress]{natbib}%
\providecommand{\U}[1]{\protect\rule{.1in}{.1in}}

\newtheorem{theorem}{Theorem}[section]

\newtheorem{definition}[theorem]{Definition}
\newtheorem{assumption}[theorem]{Assumption}

\newtheorem{lemma}[theorem]{Lemma}

\newtheorem{proposition}[theorem]{Proposition}
\newtheorem{remark}[theorem]{Remark}

\newenvironment{proof}[1][Proof]{\noindent\textbf{#1.} }{\ \rule{0.5em}{0.5em}}
\numberwithin{equation}{section}

\begin{document}

\title{A stochastic recursive optimal control problem under the G-expectation framework}
\author{Mingshang Hu \thanks{School of Mathematics, Shandong University,
humingshang@sdu.edu.cn. Research supported by the National Natural Science
Foundation of China (11201262)}
\and Shaolin Ji\thanks{Qilu Institute of Finance, Shandong University,
jsl@sdu.edu.cn. Hu, Ji, and Yang's research was partially supported by NSF of
China No. 10921101; and by the 111 Project No. B12023. }
\and Shuzhen Yang\thanks{School of mathematics, Shandong University, Jinan,
Shandong 250100, PR China. yangsz@mail.sdu.edu.cn. \textbf{This paper was
submitted to SIAM Journal on Control and Optimization on March 5, 2013
(Submitted manuscript No. 091225.)}}}
\date{}
\maketitle

\textbf{Abstract}. In this paper, we study a stochastic recursive optimal
control problem in which the objective functional is described by the solution
of a backward stochastic differential equation driven by $G$-Brownian motion.
Under standard assumptions, we establish the dynamic programming principle and
the related Hamilton-Jacobi-Bellman (HJB) equation in the framework of
$G$-expectation. Finally, we show that the value function is the viscosity
solution of the obtained HJB equation.

{\textbf{Key words}. }$G$-expectation, {backward stochastic differential
equations, stochastic optimal control, dynamic programming principle,
viscosity solution}

\textbf{AMS subject classifications.} 93E20, 60H10, 35K15

\addcontentsline{toc}{section}{\hspace*{1.8em}Abstract}

\section{Introduction}

It is well known that the nonlinear backward stochastic differential equation
(BSDE) was first introduced by Pardoux and Peng \cite{PP90}. Independently,
Duffie and Epstein \cite{DE} presented a stochastic differential recursive
utility which corresponds to the solution of a particular BSDE. Then the BSDE
point of view gives a simple formulation of recursive utilities (see
\cite{EPQ}).

Since then, the classical stochastic optimal control problem is generalized to
a so called "stochastic recursive optimal control problem" in which the cost
functional is described by the solution of BSDE. Peng \cite{peng-dpp} obtained
the Hamilton--Jacobi--Bellman equation for this kind of problem and proved
that the value function is its viscosity solution. In \cite{peng-dpp-1}, Peng
generalized his results and originally introduced the notion of stochastic
backward semigroups which allows him to prove the dynamic programming
principle in a very straightforward way. This backward semigroup approach is
proved to be a useful tool for the stochastic optimal control problems. For
instance, Wu and Yu \cite{WY} adopted this approach to study one kind of
stochastic recursive optimal control problem with the cost functional
described by the solution of a reflected BSDE. It is also introduced in the
theory of stochastic differential games by Buckdahn and Li in \cite{BL}. We
emphasize that Buckdahn et al. \cite{BLRT} obtained an existence result of the
stochastic recursive optimal control problem.

Motivated by measuring risk and other financial problems with uncertainty,
Peng \cite{P07a} introduced the notion of sublinear expectation space, which
is a generalization of probability space. As a typical case, Peng studied a
fully nonlinear expectation, called $G$-expectation $\mathbb{\hat{E}%
}\mathcal{[\cdot]}$ (see \cite{P10} and the references therein), and the
corresponding time-conditional expectation $\mathbb{\hat{E}}_{t}%
\mathcal{[\cdot]}$ on a space of random variables completed under the norm
$\mathbb{\hat{E}}[|\cdot|^{p}]^{1/p}$. Under this $G$-expectation framework
($G$-framework for short) a new type of Brownian motion called $G$-Brownian
motion was constructed. The stochastic calculus with respect to the
$G$-Brownian motion has been established. The existence and uniqueness of
solution of a SDE driven by $G$-Brownian motion can be proved in a way
parallel to that in the classical SDE theory. But the solvability of BSDE
driven by $G$-Brownian motion becomes a challenging problem. For a recent
account and development of $G$-expectation theory and its applications we
refer the reader to \cite{Peng2004, Peng2005, PengICM2010, STZ, Song11,
Song12, Nu, EJ-1, EJ-2, PSZ2012}.

Let us mention that there are other recent advances and their applications in
stochastic calculus that do not require a probability space framework. Denis
and Martini \cite{DenisMartini2006} developed quasi-sure stochastic analysis,
but they did not have conditional expectation. This topic was further examined
by Denis et al. \cite{DHP11} and Soner et al. \cite{STZ1}. It is worthing to
point out that Soner et al. \cite{STZ11} have obtained a deep result of
existence and uniqueness theorem for a new type of fully nonlinear BSDE,
called 2BSDE. Various stochastic control (game) problems are investigated in
\cite{Nutz, NZ, PZ, MPP} and the applications in finance are studied in
\cite{MPZ, NN}.

Recently Hu et. al studied the following BSDE driven by $G$-Brownian motion in
\cite{HJPS} and \cite{HJPS1}:%
\begin{align*}
Y_{t}  &  =\xi+\int_{t}^{T}f(s,Y_{s},Z_{s})ds+\int_{t}^{T}g(s,Y_{s}%
,Z_{s})d\langle B\rangle_{s}\\
&  -\int_{t}^{T}Z_{s}dB_{s}-(K_{T}-K_{t}).
\end{align*}
They proved that there exists a unique triple of processes $(Y,Z,K)$ within
our $G$-framework which solves the above BSDE under a standard Lipschitz
conditions on $f(s,y,z)$ and $g(s,y,z)$ in $(y,z)$. The decreasing
$G$-martingale K is aggregated and the solution is time consistent. Some
important properties of the BSDE driven by $G$-Brownian motion such as
comparison theorem and Girsanov transformation were given in \cite{HJPS1}.

In this paper, we study a stochastic recursive optimal control problem in
which the objective functional is described by the solution of a BSDE driven
by $G$-Brownian motion. In more details, the state equation is governed by the
following controlled SDE driven by $G$-Brownian motion
\begin{align*}
dX_{s}^{t,x,u}  &  =b(s,X_{s}^{t,x,u},u_{s})ds+h_{ij}(s,X_{s}^{t,x,u}%
,u_{s})d\langle B^{i},B^{j}\rangle_{s}+\sigma(s,X_{s}^{t,x,u},u_{s})dB_{s},\\
X_{t}^{t,x,u}  &  =x.
\end{align*}
The objective functional is introduced by the solution $Y_{t}^{t,x,u}$ of the
following BSDE driven by $G$-Brownian motion at time $t$:%

\[%
\begin{array}
[c]{l}%
-dY_{s}^{t,x,u}=f(s,X_{s}^{t,x,u},Y_{s}^{t,x,u},Z_{s}^{t,x,u},u_{s}%
)ds+g_{ij}(s,X_{s}^{t,x,u},Y_{s}^{t,x,u},Z_{s}^{t,x,u},u_{s})d\langle
B^{i},B^{j}\rangle_{s}-Z_{s}^{t,x,u}dB_{s}-dK_{s}^{t,x,u},\\
Y_{T}^{t,x,u}=\Phi(X_{T}^{t,x,u}),\text{ \ \ }s\in\lbrack t,T]\text{.}%
\end{array}
\]
We define the value function of our stochastic recursive optimal control
problem as follows:
\[
V(t,x)=\underset{u(\cdot)\in\mathcal{U}[t,T]}{\text{ess}\sup}Y_{t}^{t,x,u},
\]
where the control set is in the $G$-framework.

It is well known that dynamic programming and related HJB equations is a
powerful approach to solving optimal control problems (see \cite{Fleming W.H},
\cite{J.Yong} and \cite{peng-dpp}). The objective of our paper is to establish
the dynamic programming principle and investigate the value function in
$G$-framework. The main result of this paper states that $V$ is deterministic
continuous viscosity solution of the following HJB equation%
\begin{align*}
&  \partial_{t}V(t,x)+\sup_{u\in U}H(t,x,V,\partial_{x}V,\partial_{xx}%
^{2}V,u)=0,\\
&  V(T,x)=\Phi(x),\quad\ \ x\in\mathbb{R}^{n},
\end{align*}
where%
\[%
\begin{array}
[c]{cl}%
H(t,x,v,p,A,u)= & G(F(t,x,v,p,A,u))+\langle p,b(t,x,u)\rangle+f(t,x,v,\sigma
(x,u)p,u),\\
F_{ij}(t,x,v,p,A,u)= & \langle A\sigma_{i}(t,x,u),\sigma_{j}(t,x,u)\rangle
+2\langle p,h_{ij}(t,x,u)\rangle+2g_{ij}(t,x,v,\sigma(x,u)p,u),
\end{array}
\]
$(t,x,v,p,A,u)\in\lbrack0,T]\times\mathbb{R}^{n}\times\mathbb{R}%
\times\mathbb{R}^{d}\times\mathbb{S}_{n}\times U$ and $\sigma_{i}$ is the
$i$-th column of $\sigma$.

Notice that under the $G$-framework, there is no reference probability
measure. Thus our results generalizes the results in Peng \cite{peng-dpp} and
\cite{peng-dpp-1} which was only considered in the Wiener space (corresponding
to $G$ is linear in our paper). Under a family of non-dominated probability
measures, it is far from being trivial to prove that the value function $V$ is
wellposed and deterministic. Furthermore, the BSDE driven by $G$-Brownian
motion contains the decreasing $G$-martingale $K$, which is more difficult to
deal with.

The paper is organized as follows. In section 2, we present some fundamental
results on $G$-expectation theory and formulate our stochastic recursive
optimal control problem. We establish the dynamic programming principle in
section 3. In section 4, we first derive the HJB equation and prove that the
value function is the viscosity solution of the obtained HJB equation.

\section{Preliminaries}

We review some basic notions and results of $G$-expectation, the related
spaces of random variables and the backward stochastic differential equations
driven by a $G$-Browninan motion. The readers may refer to
\cite{HJPS,P07a,P07b,P08a,P10} for more details.

Let $\Omega$ be a given set and let $\mathcal{H}$ be a vector lattice of real
valued functions defined on $\Omega$, namely $c\in\mathcal{H}$ for each
constant $c$ and $|X|\in\mathcal{H}$ if $X\in\mathcal{H}$. $\mathcal{H}$ is
considered as the space of random variables.

\begin{definition}
\label{def2.1} A sublinear expectation $\mathbb{\hat{E}}:\mathcal{H}%
\rightarrow\mathbb{R}$ satisfying the following properties: for all
$X,Y\in\mathcal{H}$,

(i) Monotonicity: If $X\geq Y$ then $\mathbb{\hat{E}}[X]\geq\mathbb{\hat{E}%
}[Y]$;

(ii) Constant preservation: $\mathbb{\hat{E}}[c]=c$;

(iii) Sub-additivity: $\mathbb{\hat{E}}[X+Y]\leq\mathbb{\hat{E}}%
[X]+\mathbb{\hat{E}}[Y]$;

(iv) Positive homogeneity: $\mathbb{\hat{E}}[\lambda X]=\lambda\mathbb{\hat
{E}}[X]$ for each $\lambda\geq0$.

$(\Omega,\mathcal{H},\mathbb{\hat{E}})$ is called a sublinear expectation space.
\end{definition}

Let $X_{1}$ and $X_{2}$ be two $n$-dimensional random vectors defined in
sublinear expectation spaces $(\Omega_{1},\mathcal{H}_{1},\mathbb{\hat{E}}%
_{1})$ and $(\Omega_{2},\mathcal{H}_{2},\mathbb{\hat{E}}_{2})$ respectively$.$
We will denote by $C_{l.Lip}(\mathbb{R}^{n})$ the space of real continuous
functions defined on $\mathbb{R}^{n}$ such that
\[
|\varphi(x)-\varphi(y)|\leq C(1+|x|^{k}+|y|^{k})|x-y|\ \text{\ for
all}\ x,y\in\mathbb{R}^{n},
\]
where $k$ and $C$ depend only on $\varphi$.

\begin{definition}
\label{def2.2} We call $X_{1}$ and $X_{2}$\ identically distributed, denoted
by $X_{1}\overset{d}{=}X_{2}$,

if for all $\varphi\in C_{l.Lip}(\mathbb{R}^{n}),$
\[
\mathbb{\hat{E}}_{1}[\varphi(X_{1})]=\mathbb{\hat{E}}_{2}[\varphi(X_{2})].
\]
\ \
\end{definition}

\begin{definition}
\label{def2.3} For given $(\Omega,\mathcal{H},\mathbb{\hat{E}}),$ random
vectors $Y=(Y_{1},\cdots,Y_{n})$ and $X=(X_{1},\cdots,X_{m})$, $Y_{i}$,
$X_{i}\in\mathcal{H}$. We call $Y$ is independent of $X$ under $\mathbb{\hat
{E}}[\cdot]$, denoted by $Y\bot X$, if for every test function $\varphi\in
C_{l.Lip}(\mathbb{R}^{m}\times\mathbb{R}^{n})$ we have%
\[
\mathbb{\hat{E}}[\varphi(X,Y)]=\mathbb{\hat{E}}[\mathbb{\hat{E}}%
[\varphi(x,Y)]_{x=X}].
\]

\end{definition}

\begin{definition}
\label{def2.4} ($G$-normal distribution) For given $(\Omega,\mathcal{H}%
,\mathbb{\hat{E}})$ and $X=(X_{1},\cdots,X_{d}).$ $X$\ is called $G$-normally
distributed if for each $a,b\geq0,$ we have
\[
aX+b\bar{X}\overset{d}{=}\sqrt{a^{2}+b^{2}}X,
\]
where $\bar{X}$ is an independent copy of $X$, i.e., $\bar{X}\overset{d}{=}X$
and $\bar{X}\bot X$.
\end{definition}

For each $\varphi\in C_{l.Lip}(\mathbb{R}^{d})$, we define%

\[
u(t,x):=\mathbb{\hat{E}}[\varphi(x+\sqrt{t}X)],\text{ \ \ }(t,x)\in
\lbrack0,\infty)\times\mathbb{R}^{d}.
\]
Peng \cite{P10} proved that $X$ is $G$-normally distributed if and only if $u$
is the solution of the following $G$-heat equation:%
\[
\partial_{t}u-G(D_{xx}^{2}u)=0,\ u(0,x)=\varphi(x)
\]
where $G$ denotes the function
\[
G(A):=\frac{1}{2}\mathbb{\hat{E}}[\langle AX,X\rangle]:\mathbb{S}%
_{d}\rightarrow\mathbb{R}.
\]
The function $G(\cdot):\mathbb{S}_{d}\rightarrow\mathbb{R}$ is a monotonic,
sublinear mapping on $\mathbb{S}_{d},$ where $\mathbb{S}_{d}$ denotes the
collection of $d\times d$ symmetric matrices. There exists a bounded and
closed subset $\Gamma\subset$$\mathbb{R}$$^{d\times d}$ such that
\begin{equation}
G(A)=\frac{1}{2}\sup_{\gamma\in\Gamma}\text{\textrm{tr}}[\gamma\gamma^{T}A],
\label{Gequation}%
\end{equation}
where $\mathbb{R}$$^{d\times d}$ denotes the collection of $d\times d$ matrices.

In this paper we only consider non-degenerate $G$-normal distribution, i.e.,
there exists some $\underline{\sigma}^{2}>0$ such that $G(A)-G(B)\geq
\underline{\sigma}^{2}\mathrm{tr}[A-B]$ for any $A\geq B$.

Let $\Omega=C_{0}([0,\infty);\mathbb{R}^{d})$ be the space of real valued
continuous functions on $[0,\infty)$ with $\omega_{0}=0$ and let $B_{t}%
(\omega)=\omega_{t}$ be the canonical process. Set%
\[
L_{ip}(\Omega):=\{\varphi(B_{t_{1}},\cdots,B_{t_{n}}):n\geq1,t_{1}%
,\cdots,t_{n}\in\lbrack0,\infty),\varphi\in C_{l.Lip}(\mathbb{R}^{d\times
n})\}.
\]

Let $\{\xi_{n}:n\geq1\}$ be a sequence of identically distributed
$d$-dimensional $G$-normally distributed random vectors in a sublinear
expectation space $(\tilde{\Omega},\tilde{\mathcal{H}},\mathbb{\tilde{E}})$
such that $\xi_{i+1}$ is independent of $(\xi_{1},\cdot\cdot\cdot,\xi_{i})$
for every $i\geq1$.

\begin{definition}
\label{def2.5}For each $X=\varphi(B_{t_{1}}-B_{t_{0}},B_{t_{2}}-B_{t_{1}%
},\cdots,B_{t_{m}}-B_{t_{m-1}})\in L_{ip}(\Omega)$ with $0\leq t_{0}%
<\cdots<t_{m}$, the $G$-expectation of $X$ is defined by
\[
\mathbb{\hat{E}}[X]=\mathbb{\tilde{E}}[\varphi(\sqrt{t_{1}-t_{0}}\xi
_{1},\cdots,\sqrt{t_{m}-t_{m-1}}\xi_{m})].
\]
The conditional $G$-expectation $\mathbb{\hat{E}}_{t}$ of $X$ with $t=t_{i}$
is defined by
\[
\mathbb{\hat{E}}_{t_{i}}[\varphi(B_{t_{1}}-B_{t_{0}},B_{t_{2}}-B_{t_{1}}%
,\cdot\cdot\cdot,B_{t_{m}}-B_{t_{m-1}})]
\]%
\[
=\tilde{\varphi}(B_{t_{1}}-B_{t_{0}},B_{t_{2}}-B_{t_{1}},\cdot\cdot
\cdot,B_{t_{i}}-B_{t_{i-1}}),
\]
where
\[
\tilde{\varphi}(x_{1},\cdot\cdot\cdot,x_{i})=\mathbb{\hat{E}}[\varphi
(x_{1},\cdot\cdot\cdot,x_{i},B_{t_{i+1}}-B_{t_{i}},\cdot\cdot\cdot,B_{t_{m}%
}-B_{t_{m-1}})].
\]
$(\Omega,L_{ip}(\Omega),\mathbb{\hat{E}})$ is called a $G$-expectation space.
The corresponding canonical process $(B_{t})_{t\geq0}$ is called a
$G$-Brownian motion.
\end{definition}

We denote by $L_{G}^{p}(\Omega)$ the completion of $L_{ip}(\Omega)$ under the
norm $\Vert X\Vert_{p,G}=(\mathbb{\hat{E}}[|X|^{p}])^{1/p}$ for $p\geq1$. For
each$\ t\geq0$, $\mathbb{\hat{E}}_{t}[\cdot]$ can be extended continuously to
$L_{G}^{1}(\Omega)$ under the norm $\Vert\cdot\Vert_{1,G}$. For each fixed
$T>0$, set%
\[
L_{ip}(\Omega_{T}):=\{\varphi(B_{t_{1}},\cdots,B_{t_{n}}):n\geq1,t_{1}%
,\cdots,t_{n}\in\lbrack0,T],\varphi\in C_{l.Lip}(\mathbb{R}^{d\times n})\}.
\]
Obviously, $L_{ip}(\Omega_{T})\subset L_{ip}(\Omega)$, then we can similarly
define $L_{G}^{p}(\Omega_{T})$ for $p\geq1$.

\begin{definition}
\label{def2.6} Let $M_{G}^{0}(0,T)$ be the collection of processes in the
following form: for a given partition $\{t_{0},\cdot\cdot\cdot,t_{N}\}=\pi
_{T}$ of $[0,T]$,
\[
\eta_{t}(\omega)=\sum_{j=0}^{N-1}\xi_{j}(\omega)I_{[t_{j},t_{j+1})}(t),
\]
where $\xi_{i}\in L_{ip}(\Omega_{t_{i}})$, $i=0,1,2,\cdot\cdot\cdot,N-1$.
\end{definition}

We denote by $M_{G}^{p}(0,T)$ the completion of $M_{G}^{0}(0,T)$ under the
norm $\Vert\eta\Vert_{M_{G}^{p}}=\{\mathbb{\hat{E}}[\int_{0}^{T}|\eta_{s}%
|^{p}ds]\}^{1/p}$ for $p\geq1$.

\begin{theorem}
\label{the2.7} (\cite{DHP11,HP09}) There exists a family of weakly compact
probability measures $\mathcal{P}$ on $(\Omega,\mathcal{B}(\Omega))$ such
that
\[
\mathbb{\hat{E}}[\xi]=\sup_{P\in\mathcal{P}}E_{P}[\xi]\ \ \text{for
\ all}\ \xi\in L_{G}^{1}(\Omega).
\]
$\mathcal{P}$ is called a set that represents $\mathbb{\hat{E}}$.
\end{theorem}

Let $\{W_{t}\}$ be a classical $d$-dimensional Brownian motion on a
probability space $(\Omega^{0},\mathcal{F}^{0},P^{0})$ and let$\ F^{0}%
=\{\mathcal{F}_{t}^{0}\}$ be the augmented filtration generated by $W$. Set
\[
\mathcal{P}_{M}:=\{P_{\theta}:P_{\theta}=P^{0}\circ(B_{t}^{\theta,0}%
)^{-1},B_{t}^{\theta,0}=\int_{0}^{t}\theta_{s}dW_{s},\theta\in L_{F^{0}}%
^{2}([0,T];\Gamma)\},
\]
where $L_{F^{0}}^{2}([0,T];\Gamma)$ is the collection of $F^{0}$-adapted
square integrable measurable processes with values in $\Gamma$. Set
$\mathcal{P=}\overline{\mathcal{P}_{M}}\mathcal{\ }$the closure of
$\mathcal{P}_{M}$ under the topology of weak convergence, then $\mathcal{P}$
is weakly compact. \cite{DHP11} proved that $\mathcal{P}$ represents
$\mathbb{\hat{E}}$ on $L_{G}^{1}(\Omega_{T})$.

\begin{proposition}
\label{npro-2.8} (\cite{DHP11}) Let $\{P_{n}:n\geq1\}\subset\mathcal{P}$
converge weakly to $P$. Then for each $\xi\in L_{G}^{1}(\Omega)$, we have
$E_{P_{n}}[\xi]\rightarrow E_{P}[\xi]$.
\end{proposition}

For this $\mathcal{P}$, we define capacity%
\[
c(A):=\sup_{P\in\mathcal{P}}P(A),\ A\in\mathcal{B}(\Omega).
\]
A set $A\in\mathcal{B}(\Omega)$ is polar if $c(A)=0$. A property holds
\textquotedblleft quasi-surely\textquotedblright\ (q.s. for short) if it holds
outside a polar set. In the following, we do not distinguish two random
variables $X$ and $Y$ if $X=Y$ q.s.. We set%
\[
\mathbb{L}^{p}(\Omega_{T}):=\{X\in\mathcal{B}(\Omega_{T}):\sup_{P\in
\mathcal{P}}E_{P}[|X|^{p}]<\infty\}\ \text{for}\ p\geq1.
\]
It is important to note that $L_{G}^{p}(\Omega_{T})\subset\mathbb{L}%
^{p}(\Omega_{T})$.

\subsection{Forward and backward SDEs driven by $G$-Brownian motion}

We first give the definition of admissible controls.

\begin{definition}
For each $t\in\lbrack0,T],$ $u$ is said to be an admissible control on
$[t,T]$, if it satisfies the following conditions:

(i) $u:[t,T]\times\Omega\rightarrow U$ where $U$ is a compact set of
$\mathbb{R}^{m}$;

(ii) $u\in M_{G}^{2}(t,T;\mathbb{R}^{m})$.
\end{definition}

The set of admissible controls on $[t,T]$ is denoted by $\mathcal{U}[t,T]$.

In the rest of this paper, we use Einstein summation convention.

Let $t\in\lbrack0,T]$, $\varepsilon>0$, $\xi\in L_{G}^{2+\varepsilon}%
(\Omega_{t})$ and $u\in\mathcal{U}[t,T]$. Consider the following forward and
backward SDEs driven by $G$-Brownian motion:
\begin{align}
dX_{s}^{t,\xi,u}  &  =b(s,X_{s}^{t,\xi,u},u_{s})ds+h_{ij}(s,X_{s}^{t,\xi
,u},u_{s})d\langle B^{i},B^{j}\rangle_{s}+\sigma(s,X_{s}^{t,\xi,u}%
,u_{s})dB_{s},\label{state-1}\\
X_{t}^{t,\xi,u}  &  =\xi,\nonumber
\end{align}

and%
\begin{equation}%
\begin{array}
[c]{l}%
-dY_{s}^{t,\xi,u}=f(s,X_{s}^{t,\xi,u},Y_{s}^{t,\xi,u},Z_{s}^{t,\xi,u}%
,u_{s})ds+g_{ij}(s,X_{s}^{t,\xi,u},Y_{s}^{t,\xi,u},Z_{s}^{t,\xi,u}%
,u_{s})d\langle B^{i},B^{j}\rangle_{s}-Z_{s}^{t,\xi,u}dB_{s}-dK_{s}^{t,\xi
,u},\\
Y_{T}^{t,\xi,u}=\Phi(X_{T}^{t,\xi,u}),\text{ \ \ \ }s\in\lbrack t,T]\text{,}%
\end{array}
\label{state-2}%
\end{equation}
where
\[%
\begin{array}
[c]{l}%
b:[t,T]\times\mathbb{R}^{n}\times U\rightarrow\mathbb{R}^{n}\text{;}\\
h_{ij}:[t,T]\times\mathbb{R}^{n}\times U\rightarrow\mathbb{R}^{n}\text{;}\\
\sigma:[t,T]\times\mathbb{R}^{n}\times U\rightarrow\mathbb{R}^{n\times
d}\text{;}\\
f:[t,T]\times\mathbb{R}^{n}\times\mathbb{R}\times\mathbb{R}^{d}\times
U\rightarrow\mathbb{R}\text{;}\\
g_{ij}:[t,T]\times\mathbb{R}^{n}\times\mathbb{R}\times\mathbb{R}^{d}\times
U\rightarrow\mathbb{R}\text{;}\\
\Phi:\mathbb{R}^{n}\rightarrow\mathbb{R}.
\end{array}
\]

Denote%
\[%
\begin{array}
[c]{l}%
S_{G}^{0}(0,T)=\{h(t,B_{t_{1}\wedge t},\cdot\cdot\cdot,B_{t_{n}\wedge
t}):t_{1},\ldots,t_{n}\in\lbrack0,T],h\in C_{b,Lip}(\mathbb{R}^{n+1})\};\\
S_{G}^{2}(0,T)=\{\text{the completion of }S_{G}^{0}(0,T)\text{ under the norm
}\Vert\eta\Vert_{S_{G}^{2}}=\{\mathbb{\hat{E}}[\sup_{t\in\lbrack0,T]}|\eta
_{t}|^{2}]\}^{\frac{1}{2}}\}.
\end{array}
\]

For given $t,$ $u$ and $\xi$, $(X^{t,\xi,u})$ and $(Y^{t,\xi,u},Z^{t,\xi
,u},K^{t,\xi,u})$ are called solutions of the above forward and backward SDEs
respectively if $(X^{t,\xi,u})\in M_{G}^{2}(t,T;\mathbb{R}^{n})$;
$(Y^{t,\xi,u},Z^{t,\xi,u})\in S_{G}^{2}(0,T)\times M_{G}^{2}(0,T)$;
$K^{t,\xi,u}$ is a decreasing $G$-martingale$\ $with $K_{t}^{t,\xi,u}=0,$
$K_{T}^{t,\xi,u}\in L_{G}^{2}(\Omega_{T});$ (\ref{state-1}) and (\ref{state-2}%
) are satisfied respectively.

Assume $b,h_{ij},\sigma,f,g_{ij},\Phi$ are deterministic functions and
satisfying the following conditions:

\begin{assumption}
\label{assu-1}There exists a constant $c>0$ such that%
\[%
\begin{array}
[c]{cl}
& \mid b(s,x^{1},u)-b(s,x^{2},v)\mid+\mid h_{ij}(s,x^{1},u)-h_{ij}%
(s,x^{2},v)\mid+\mid\sigma(s,x^{1},u)-\sigma(s,x^{2},v)\mid\\
\leq & c(\mid x^{1}-x^{2}\mid+\mid u-v\mid),\text{ }\forall(s,x^{1}%
,u),(s,x^{2},v)\in\lbrack t,T]\times\mathbb{R}^{n}\times U
\end{array}
\]
and $b,h_{ij},\sigma$\ are continuous\ about $t$.
\end{assumption}

\begin{assumption}
\label{assu-2}There exists a constant $c>0$ such that%
\[%
\begin{array}
[c]{l}%
\mid f(s,x^{1},y^{1},z^{1},u)-f(s,x^{2},y^{2},z^{2},v)\mid\leq c(\mid
x^{1}-x^{2}\mid+\mid y^{1}-y^{2}\mid+\mid z^{1}-z^{2}\mid+|u-v|);\\
\mid g_{ij}(s,x^{1},y^{1},z^{1},u)-g_{ij}(s,x^{2},y^{2},z^{2},v)\mid\leq
c(\mid x^{1}-x^{2}\mid+\mid y^{1}-y^{2}\mid+\mid z^{1}-z^{2}\mid+|u-v|);\\
\mid\Phi(x^{1})-\Phi(x^{2})\mid\leq c\mid x^{1}-x^{2}\mid,\\
\forall(s,x^{1},y^{1},z^{1},u),(s,x^{2},y^{2},z^{2},v)\in\lbrack
t,T]\times\mathbb{R}^{n}\times\mathbb{R}\times\mathbb{R}^{d}\times U
\end{array}
\]
and $f,g_{ij}$\ are continuous about $t$.\
\end{assumption}

\begin{remark}
Suppose Assumptions (\ref{assu-1}) and (\ref{assu-2}) hold. Then there exists
a constant $K>0$ such that%
\[%
\begin{array}
[c]{l}%
\mid b(s,x,u)\mid+\mid h_{ij}(s,x,u)\mid+\mid\sigma(s,x,u)\mid\leq K(1+\mid
x\mid);\\
\mid f(s,x,0,0,u)\mid+\mid g_{ij}(s,x,0,0,u)\mid\leq c(1+\mid x\mid);\\
\mid\Phi(x)\mid\leq K(1+\mid x\mid),\text{ \ \ }\forall(s,x,u)\in\lbrack
t,T]\times\mathbb{R}^{n}\times U.
\end{array}
\]

\end{remark}

We have the following theorems.

\begin{theorem}
(\cite{P10}) Let Assumption \ref{assu-1} hold. Then there exists a unique
adapted solution $X$ for equation (\ref{state-1}).
\end{theorem}

\begin{theorem}
(\cite{HJPS1}) Let Assumption \ref{assu-2} hold. Then there exists a unique
adapted solution $(Y,Z,K)$ for equation (\ref{state-2}).
\end{theorem}

\subsection{Stochastic optimal control problem}

The state equation of our stochastic optimal control problem is governed by
the above forward SDE (\ref{state-1}) and the objective functional is
introduced by the solution of the BSDE (\ref{state-2}) at time $t$. Let $\xi$
equals a constant $x\in\mathbb{R}^{n}$. When $u$ changes, $Y_{t}^{t,x,u}$ (the
solution $Y^{t,x,u}$ at time $t$) also changes. In order to study the value
function of our stochastic optimal control problem, we need to define the
essential supremum of $\{Y_{t}^{t,x,u}\mid u\in\mathcal{U}[t,T]\}.$

\begin{definition}
\label{esssup}The essential supremum of $\{Y_{t}^{t,x,u}\mid u\in
\mathcal{U}[t,T]\}$, denoted by $\underset{u(\cdot)\in\mathcal{U}%
[t,T]}{\text{ess}\sup}Y_{t}^{t,x,u}$, is a random variable $\zeta\in L_{G}%
^{2}(\Omega_{t})$ satisfying:

(i). $\forall u\in\mathcal{U}[t,T],$ $\zeta\geq Y_{t}^{t,x,u}$ $\ q.s.,$ and

(ii). if $\eta$ is a random variable satisfying $\eta\geq Y_{t}^{t,x,u}$
$\ q.s.$ for any $u\in\mathcal{U}[t,T]$, then $\zeta\leq\eta$ $\ q.s..$
\end{definition}

\begin{remark}
It is easy to verify that $c(A)=0$ if and only if $P(A)=0$ for each
$P\in\mathcal{P}$. Thus $\zeta\leq\eta$ $\ q.s.$ is equivalent to $\zeta
\leq\eta$ $P-a.s.$ for each $P\in\mathcal{P}$.
\end{remark}

\begin{proposition}
Let $\zeta$, $\eta\in L_{G}^{2}(\Omega)$. If $\zeta\leq\eta$ $P-a.s.$ for each
$P\in\mathcal{P}_{M}$, then $\zeta\leq\eta$ $\ q.s.$.
\end{proposition}

\begin{proof}
It is easy to check that $(\zeta-\eta)^{+}\in L_{G}^{2}(\Omega)$. By
Proposition \ref{npro-2.8}, we obtain
\[
\mathbb{\hat{E}}[(\zeta-\eta)^{+}]=\sup_{P\in\mathcal{P}}E_{P}[(\zeta
-\eta)^{+}]=\sup_{P\in\mathcal{P}_{M}}E_{P}[(\zeta-\eta)^{+}]=0.
\]
Thus $\zeta\leq\eta$ $\ q.s.$.
\end{proof}

\begin{remark}
\label{ren1} From the above proposition, it is easy to deduce that $\zeta
\leq\eta$ $\ q.s.$ if and only if $\zeta\leq\eta$ $P-a.s.$ for each
$P\in\mathcal{P}_{M}$.
\end{remark}

Our stochastic optimal control problem is: for given $x\in\mathbb{R}^{n}$, to
find $u(\cdot)\in\mathcal{U}[t,T]$ so as to maximize the objective function
$Y_{t}^{t,x,u}$.

The value function $V$ is defined to be
\begin{equation}
V(t,x):=\underset{u(\cdot)\in\mathcal{U}[t,T]}{\text{ess}\sup}Y_{t}^{t,x,u}.
\label{valuefunction0}%
\end{equation}
Next we prove that $V(t,x)$ exists and is deterministic, and then we show that
it satisfies a kind of HJB equation.

For $x\in R^{n}$, $u(\cdot)\in\mathcal{U}[t,T]$ and $P\in\mathcal{P}_{M},$ we
consider the following forward and backward equation:%
\begin{align}
&  dX_{s}^{t,x,u;P}=b(s,X_{s}^{t,x,u;P},u_{s})ds+h_{ij}(s,X_{s}^{t,x,u;P}%
,u_{s})d\langle B^{i},B^{j}\rangle_{s}+\sigma(s,X_{s}^{t,x,u;P},u_{s}%
)dB_{s},\label{socp-1}\\
&  X_{t}^{t,x,u;P}=x,\ \ \ P-a.s.\nonumber
\end{align}

and%

\begin{equation}%
\begin{array}
[c]{l}%
dY_{s}^{t,x,u;P}=-f(s,X_{s}^{t,x,u;P},Y_{s}^{t,x,u;P},Z_{s}^{t,x,u;P}%
,u_{s})ds-g_{ij}(s,X_{s}^{t,x,u;P},Y_{s}^{t,x,u;P},Z_{s}^{t,x,u;P}%
,u_{s})d\langle B^{i},B^{j}\rangle_{s}+Z_{s}^{t,x,u;P}dB_{s},\\
Y_{T}^{t,x,u;P}=\Phi(X_{T}^{t,x,u;P}),\text{ \ \ }s\in\lbrack t,T],\text{
\ \ }P-a.s.\text{.}%
\end{array}
\label{scop-2}%
\end{equation}

\begin{remark}
Note that under probability $P\in\mathcal{P}_{M}$, the process $\{B_{s}%
\}_{t\leq s\leq T}$ in the\ equation (\ref{socp-1}) and (\ref{scop-2}) is
generally not a standard Brownian Motion. But the martingale representation
property still holds for $P$ (see \cite{STZ} and \cite{STZ11}), thus there
still exist unique solutions for (\ref{socp-1}) and (\ref{scop-2}).
\end{remark}

By \cite{P10}, we have
\[
X_{s}^{t,x,u;P}=X_{s}^{t,x,u}\text{ \ }P-a.s..
\]
Soner et al. \cite{STZ11} give the following representation for the solution
$Y^{t,x,u}$\ of (\ref{state-2}):%
\[
Y_{t}^{t,x,u}=\underset{Q\in\mathcal{P}_{M}(t,P)}{\text{ess}\sup}^{P}\text{
}Y_{t}^{t,x,u;Q},\text{ \ }\ P-a.s.,
\]
where ess$\sup^{P}$ is the ess$\sup$ with respect to probability $P$ in the
classical sense and
\[
\mathcal{P}_{M}(t,P):=\{Q:Q(A)=P(A),\forall A\in\mathcal{F}_{t},Q\in
\mathcal{P}_{M}\}.
\]

For each fix $P\in\mathcal{P}_{M}$, the value function $V^{P}$ is defined to
be
\begin{equation}
V^{P}(t,x):=\underset{u(\cdot)\in\mathcal{U}[t,T]}{\text{ess}\sup}\text{
}\underset{Q\in\mathcal{P}_{M}(t,P)}{\text{ess}\sup}^{P}\text{ }%
Y_{t}^{t,x,u;Q},\text{ \ }\ P-a.s.. \label{value function}%
\end{equation}

\begin{remark}
If $V^{P}(t,x)$ is a deterministic function and independent of $P$, then by
Remark \ref{ren1}, we have $V(t,x)=V^{P}(t,x)$.
\end{remark}

\section{Dynamic programming principle}

For given initial data $(t,x)$, a positive real number $\delta\leq T-t$ and
$\eta\in L_{G}^{2}(\Omega_{t+\delta})$, we define
\[
\mathbb{G}_{t,t+\delta}^{t,x,u}[\eta]:=Y_{t}^{t,x,u},
\]
where $(X_{s}^{t,x,u},Y_{s}^{t,x,u},Z_{s}^{t,x,u})_{t\leq s\leq t+\delta}$ is
the solution of the following forward and backward equations:
\begin{align*}
dX_{s}^{t,x,u}  &  =b(s,X_{s}^{t,x,u},u_{s})ds+h_{ij}(s,X_{s}^{t,x,u}%
,u_{s})d\langle B^{i},B^{j}\rangle_{s}+\sigma(s,X_{s}^{t,x,u},u_{s})dB_{s},\\
X_{t}^{t,x,u}  &  =x
\end{align*}

and%
\begin{equation}%
\begin{array}
[c]{l}%
-dY_{s}^{t,x,u}=f(s,X_{s}^{t,x,u},Y_{s}^{t,x,u},Z_{s}^{t,x,u},u_{s}%
)ds+g_{ij}(s,X_{s}^{t,x,u},Y_{s}^{t,x,u},Z_{s}^{t,x,u},u_{s})d\langle
B^{i},B^{j}\rangle_{s}-Z_{s}^{t,x,u}dB_{s}-dK_{s}^{t,x,u},\\
Y_{t+\delta}^{t,x,u}=\eta,\text{ \ \ \ }s\in\lbrack t,t+\delta]\text{.}%
\end{array}
\label{prob--state-2}%
\end{equation}
Note that $\mathbb{G}_{t,t+\delta}^{t,x,u}[\cdot]$ is a (backward) semigroup
which was first introduced by Peng in\ \cite{peng-dpp-1}.

Now we give some notations:%
\[%
\begin{array}
[c]{l}%
L_{ip}(\Omega_{s}^{t}):=\{\varphi(B_{t_{1}}-B_{t},...,B_{t_{n}}-B_{t}%
):n\geq1,t_{1},...,t_{n}\in\lbrack t,s],\varphi\in C_{l.Lip}(\mathbb{R}%
^{d\times n})\};\\
M_{G}^{0,t}(t,T):=\{\eta_{s}=\sum_{j=0}^{N-1}\xi_{j}I_{[t_{j},t_{j+1}%
)}(s):s\in\lbrack t,T],t=t_{0}<\cdots<t_{N}=T,\xi_{i}\in L_{ip}(\Omega_{t_{i}%
}^{t})\};\\
M_{G}^{2,t}(t,T):=\{\text{the completion of }M_{G}^{0,t}(t,T)\text{ under
}\Vert\cdot\Vert_{M_{G}^{2}}\};\\
\mathcal{U}^{t}[t,T]:=\{u\in M_{G}^{2,t}(t,T;\mathbb{R}^{m})\text{ with values
in }U\};\\
\mathcal{U}_{0}[t,T]:=\{u=\sum\limits_{i=1}^{m}1_{A_{i}}u^{i}:m\in
\mathbb{N},u^{i}\in\mathcal{U}^{t}[t,T],\text{where }\{A_{i}\}_{i=1,\ldots
m}\text{ is a partition of }\Omega,A_{i}\in\mathcal{B}(\Omega_{s})\}.
\end{array}
\]

Our main result in this section is the following\ dynamic programming principle.

\begin{theorem}
\label{Thm-dpp} Let Assumptions \ref{assu-1} and \ref{assu-2} hold. Then for
any $t\leq T$, $x\in\mathbb{R}^{n}$, $V(t,x)$ exists and is deterministic.
Furthermore, for any $s\in\lbrack t,T]$, we have
\begin{equation}%
\begin{array}
[c]{rl}%
V(t,x)= & \underset{u(\cdot)\in\mathcal{U}[t,s]}{\text{ess}\sup}%
\mathbb{G}_{t,s}^{t,x,u}[V(s,X_{s}^{t,x,u})]\\
= & \underset{u(\cdot)\in\mathcal{U}^{t}[t,s]}{\sup}\mathbb{G}_{t,s}%
^{t,x,u}[V(s,X_{s}^{t,x,u})].
\end{array}
\label{DPP}%
\end{equation}

\end{theorem}

In order to prove Theorem \ref{Thm-dpp}, we need to study $V^{P}(t,x)$ through
equations (\ref{socp-1}) and (\ref{scop-2}). The following priori estimates
are classical and we omit the proof (refer to \cite{EPQ}).

\begin{lemma}
\label{Lem-basic est-0} Under Assumptions \ref{assu-1} and \ref{assu-2}, for
any $\xi_{1},\xi_{2}\in L_{G}^{2}(\Omega_{t})$ and $u,v\in\mathcal{U}[t,T],$
there exists a constant $C_{2}$ such that%
\[%
\begin{array}
[c]{l}%
E_{P}[\sup_{s\in\lbrack t,T]}\mid X_{s}^{t,\xi_{1},u;P}-X_{s}^{t,\xi_{2}%
,u;P}\mid^{2}\mid\mathcal{F}_{t}]\leq C_{2}\mid\xi_{1}-\xi_{2}\mid^{2};\\
\mid Y_{t}^{t,\xi_{1},u;P}-Y_{t}^{t,\xi_{2},u;P}\mid^{2}\leq C_{2}\mid\xi
_{1}-\xi_{2}\mid^{2};\\
\mid Y_{t}^{t,\xi_{1},u;P}-Y_{t}^{t,\xi_{1},v;P}\mid^{2}\leq C_{2}E_{P}%
[\int_{t}^{T}\mid u_{s}-v_{s}\mid^{2}ds\mid\mathcal{F}_{t}],\text{
\ \ }P-a.s..
\end{array}
\]

\end{lemma}

The following theorem shows that $V^{P}(t,x)$ is deterministic and independent
of $P\in\mathcal{P}_{M}$.

\begin{theorem}
\label{Compare}Under Assumptions (\ref{assu-1}) and (\ref{assu-2}), we have

(i) For a fixed $P\in\mathcal{P}_{M}$, $V^{P}(t,x)$ is a deterministic
function and
\[
V^{P}(t,x)=\underset{u\in\mathcal{U}^{t}[t,T]}{\text{ess}\sup}\underset{Q\in
\mathcal{P}_{M}(t,P)}{\text{ess}\sup}^{P}\text{ }Y_{t}^{t,x,u;Q},\ P-a.s.;
\]

(ii) For each $u\in\mathcal{U}^{t}[t,T],$ $Y_{t}^{t,x,u}$\ (the solution of
(\ref{prob--state-2}) at time $t$) is a deterministic function. Furthermore,
\[
V(t,x)=\underset{u\in\mathcal{U}^{t}[t,T]}{\sup}Y_{t}^{t,x,u}.
\]

\end{theorem}

\noindent\textbf{Proof:} \textbf{(i)} Without loss of generality, for
(\ref{socp-1}) and (\ref{scop-2}), we only study the case $n=d=1$ and
$h_{ij}=g_{ij}=0$.

By the definition of $V^{P}(t,x),$%

\[
V^{P}(t,x)\underset{u\in\mathcal{U}^{t}[t,T]}{\geq\text{ess}\sup}\text{
}\underset{Q\in\mathcal{P}_{M}(t,P)}{\text{ess}\sup}^{P}\text{ }%
Y_{t}^{t,x,u;Q}.
\]

Analysis similar to that in Lemma 43 in \cite{DHP11} shows that $\mathcal{U}%
_{0}[t,T]$ is dense in $\mathcal{U}[t,T]$ under probability $Q$. It yields that%

\begin{equation}%
\begin{array}
[c]{rc}%
V^{P}(t,x)= & \underset{u(\cdot)\in\mathcal{U}[t,T]}{\text{ess}\sup}\text{
}\underset{Q\in\mathcal{P}_{M}(t,P)}{\text{ess}\sup}^{P}\text{ }%
Y_{t}^{t,x,u;Q}\\
= & \underset{Q\in\mathcal{P}_{M}(t,P)}{\text{ess}\sup}^{P}\text{
}\underset{u(\cdot)\in\mathcal{U}[t,T]}{\text{ess}\sup}\text{ }Y_{t}%
^{t,x,u;Q}\\
= & \underset{Q\in\mathcal{P}_{M}(t,P)}{\text{ess}\sup}^{P}\text{
}\underset{u\in\mathcal{U}_{0}[t,T]}{\text{ess}\sup}\text{ }Y_{t}^{t,x,u;Q}\\
= & \underset{u\in\mathcal{U}_{0}[t,T]}{\text{ess}\sup}\text{ }\underset{Q\in
\mathcal{P}_{M}(t,P)}{\text{ess}\sup}^{P}\text{ }Y_{t}^{t,x,u;Q}.
\end{array}
\label{deter-2}%
\end{equation}

Set $u=\sum\limits_{i=1}^{m}1_{A_{i}}u^{i}\in\mathcal{U}_{0}[t,T].$ Consider
the following equation%

\[%
\begin{array}
[c]{l}%
X_{s}^{{t,x},u^{i};Q}=x+\int_{t}^{s}b(r,X_{r}^{{t,x},u^{i};Q},u_{r}%
^{i})dr+\int_{t}^{s}\sigma(r,X_{r}^{{t,x},u^{i};Q},u_{r}^{i})dB_{r},\\
Y_{s}^{{t,x},u^{i};Q}=\Phi(X_{T}^{{t,x},u^{i};Q})+\int_{s}^{T}f(r,X_{r}%
^{{t,x},u^{i};Q},Y_{r}^{{t,x},u^{i};Q},Z_{r}^{{t,x},u^{i};Q},u_{r}^{i}%
)dr-\int_{s}^{T}Z_{r}^{{t,x},u^{i};Q}dB_{r}\text{.}%
\end{array}
\]
Multiplying by $I_{A_{i}}$ and adding the corresponding terms, we obtain%

\[%
\begin{array}
[c]{ll}%
\sum\limits_{i=1}^{N}1_{A_{i}}X_{s}^{{t,x},u^{i};Q}= & x+\sum\limits_{i=1}%
^{N}1_{A_{i}}\int_{t}^{s}b(r,X_{r}^{{t,x},u^{i};Q},u_{r}^{i})dr+\sum
\limits_{i=1}^{N}1_{A_{i}}\int_{t}^{s}\sigma(r,X_{r}^{{t,x},u^{i};Q},u_{r}%
^{i})dB_{r},\\
\sum\limits_{i=1}^{N}1_{A_{i}}Y_{s}^{{t,x},u^{i};Q}= & \sum\limits_{i=1}%
^{N}1_{A_{i}}\Phi(X_{T}^{{t,x},u^{i};Q})-\sum\limits_{i=1}^{N}1_{A_{i}}%
\int_{s}^{T}Z_{r}^{{t,x},u^{i};Q}dB_{r}\\
& +\sum\limits_{i=1}^{N}1_{A_{i}}\int_{s}^{T}f(r,X_{r}^{{t,x},u^{i};Q}%
,Y_{r}^{{t,x},u^{i};Q},Z_{r}^{{t,x},u^{i};Q},u_{r}^{i})dr.
\end{array}
\]
Then%
\[%
\begin{array}
[c]{ll}%
\sum\limits_{j=1}^{N}1_{A_{i}}X_{s}^{{t,x},u^{i};Q}= & x+\int_{t}^{s}%
b(r,\sum\limits_{j=1}^{N}1_{A_{i}}X_{r}^{{t,x},u^{i};Q},\sum\limits_{j=1}%
^{N}1_{A_{i}}u_{r}^{i})dr+\int_{t}^{s}\sigma(r,\sum\limits_{j=1}^{N}1_{A_{i}%
}X_{r}^{{t,x},u^{i};Q},\sum\limits_{j=1}^{N}1_{A_{i}}u_{r}^{i})dB_{r},\\
\sum\limits_{j=1}^{N}1_{A_{i}}Y_{s}^{{t,x},u^{i};Q}= & \Phi(\sum
\limits_{j=1}^{N}1_{A_{i}}X_{T}^{{t,x},u^{i};Q})-\int_{s}^{T}(\sum
\limits_{j=1}^{N}1_{A_{i}}Z_{r}^{{t,x},u^{i};Q})dB_{r}\\
& +\int_{s}^{T}f(r,\sum\limits_{i=1}^{N}1_{A_{i}}X_{r}^{{t,x},u^{i};Q}%
,\sum\limits_{i=1}^{N}1_{A_{i}}Y_{r}^{{t,x},u^{i};Q},\sum\limits_{i=1}%
^{N}1_{A_{i}}Z_{r}^{{t,x},u^{i};Q},\sum\limits_{i=1}^{N}1_{A_{i}}u_{r}%
^{i})dr\text{.}%
\end{array}
\]
By the uniqueness theorem of BSDE, we have%

\[
Y_{t}^{{t,x},u;Q}=\sum\limits_{i=1}^{N}1_{A_{i}}Y_{t}^{{t,x},u^{i};Q}\leq
\sum\limits_{i=1}^{N}1_{A_{i}}\underset{u\in\mathcal{U}^{t}[t,T]}{\text{ess}%
\sup}\text{ }\underset{Q\in\mathcal{P}_{M}(t,P)}{\text{ess}\sup}^{P}\text{
}Y_{t}^{t,x,u;Q}.
\]

From this we get%
\[
V^{P}(t,x)\leq\underset{u\in\mathcal{U}^{t}[t,T]}{\text{ess}\sup
}\underset{Q\in\mathcal{P}_{M}(t,P)}{\text{ess}\sup}^{P}Y_{t}^{t,x,u;Q}%
,\ P-a.s..
\]
Thus%
\[
V^{P}(t,x)=\underset{u\in\mathcal{U}^{t}[t,T]}{\text{ess}\sup}\underset{Q\in
\mathcal{P}_{M}(t,P)}{\text{ess}\sup}^{P}\text{ }Y_{t}^{t,x,u;Q}.
\]

\textbf{(ii)} For each $u\in\mathcal{U}^{t}[t,T]$, it is easy to check that
$Y_{t}^{t,x,u}$ is a deterministic function. Note that
\begin{equation}
Y_{t}^{t,x,u}=\underset{Q\in\mathcal{P}_{M}(t,P)}{\text{ess}\sup}^{P}%
Y_{t}^{t,x,u;Q}\;P-a.s.. \label{relation}%
\end{equation}
Thus for each $P\in\mathcal{P}_{M}$, we obtain%
\[
V^{P}(t,x)=\underset{u\in\mathcal{U}^{t}[t,T]}{\sup}Y_{t}^{t,x,u}\;P-a.s.,
\]
which implies that $V^{P}(t,x)$ is a constant and independent of $P$. By the
definition of $V(t,x)$, we deduce that $V(t,x)$ is deterministic and
\[
V(t,x)=\underset{u\in\mathcal{U}^{t}[t,T]}{\sup}Y_{t}^{t,x,u}.
\]

This completes the proof. \ \ \ \ $\Box$

We have the following estimations of the continuity of value function $V(t,x)$
with respect to $x$.

\begin{lemma}
\label{Lem-basic est-1}$\forall t\in\lbrack0,T]$ and $x,x^{\prime}%
\in\mathbb{R}^{n}{,}$ there exists a constant $C_{0}$ such that

(i) $\mid V(t,x)-V(t,x^{\prime})\mid\leq C_{0}\mid x-x^{\prime}\mid;$

(ii) $\mid V(t,x)\mid\leq C_{0}(1+\mid x\mid).$
\end{lemma}

\noindent\textbf{Proof. }By Lemma \ref{Lem-basic est-0}, we have
\[
\mid Y_{t}^{t,x,u;Q}-Y_{t}^{t,x^{\prime},u;Q}\mid\leq C_{0}\mid x-x^{\prime
}\mid,\text{ \ \ }P-a.s..
\]

It is easy to verify that for any $P\in\mathcal{P}_{M},$%
\[%
\begin{array}
[c]{cl}%
\mid V(t,x)-V(t,x^{\prime})\mid & \leq\underset{u(\cdot)\in\mathcal{U}%
[t,T]}{\sup}\text{ }\underset{Q\in\mathcal{P}_{M}(t,P)}{\text{ess}\sup}%
^{P}\text{ }\mid Y_{t}^{t,x,u;Q}-Y_{t}^{t,x^{\prime},u;Q}\mid\\
& \leq C_{0}\mid x-x^{\prime}\mid.
\end{array}
\]

This completes the proof. $\ \ \ \ \ \Box$

$\forall s\geq t,$ define%
\[%
\begin{array}
[c]{l}%
\mathbb{M}^{2,0}(s,T)=\{\eta_{t}=\sum_{i=0}^{N-1}\xi_{t_{i}}I_{[t_{i}%
,t_{i+1})}(t):s=t_{0}<\cdots<t_{N}=T,\xi_{t_{i}}\in\mathbb{L}^{2}%
(\Omega_{t_{i}})\};\\
\mathbb{M}_{G}^{2}(s,T)=\{\text{the completion of }\mathbb{M}^{2,0}(s,T)\text{
under }||\eta||_{\mathbb{M}^{2}}:=(\mathbb{\hat{E}}[\int_{0}^{T}|\eta_{t}%
|^{2}dt])^{1/2}\};\\
\mathbb{U}[s,T]=\{u:[s,T]\times\Omega\rightarrow U:u\in\mathbb{M}_{G}%
^{2}(s,T;\mathbb{R}^{m})\}.
\end{array}
\]

\begin{lemma}
\label{Lem-dpp ine} Suppose $s\in\lbrack0,T]$ and $\xi\in L_{G}^{2}(\Omega
_{s})$. Then we have

(i) for any $v(\cdot)\in\mathcal{U}[s,T]$ and$\ Q\in\mathcal{P}_{M}(s,P),\ $%
\begin{equation}
V(s,\xi)\geq Y_{s}^{s,\xi,v;Q},\text{ \ \ }P-a.s.; \label{dpp ine-1}%
\end{equation}

(ii) for any $\varepsilon>0,$ there is an admissible control $v^{\prime}%
(\cdot)\in\mathbb{U}[s,T]$ and$\ Q^{\prime}\in\mathcal{P}_{M}(s,P)$ such that
$\ $%
\begin{equation}
V(s,\xi)\leq Y_{s}^{s,\xi_{s},v^{\prime};Q^{\prime}}+\varepsilon,\ \ \ P-a.s.;
\label{dpp ine-2}%
\end{equation}

(iii)
\begin{equation}
V(s,\xi)=\underset{v(\cdot)\in\mathcal{U}[s,T]}{\text{ess}\sup}Y_{s}^{s,\xi
,v}. \label{dpp ine-3}%
\end{equation}

\end{lemma}

\noindent\textbf{Proof. (i) }Set$\ $%
\[
\xi=\sum\limits_{i=1}^{N}1_{A_{i}}x^{i}\in\mathbb{L}(\Omega_{s}),
\]
where $\{A_{i}\}_{i=1,\ldots N}$ is a partition of $\Omega$, $A_{i}%
\in\mathcal{B}(\Omega_{s})$ and $x^{i}\in\mathbb{R}^{n}.$

For any $v(\cdot)\in\mathcal{U}[s,T],$ $Q\in\mathcal{P}_{M}(s,P),$ we have%
\[
Y_{s}^{s,\xi,v;Q}=\sum\limits_{i=1}^{N}1_{A_{i}}Y_{s}^{s,x^{i},v;Q}\leq
\sum\limits_{i=1}^{N}1_{A_{i}}V(s,x^{i})=V(s,\sum\limits_{i=1}^{N}1_{A_{i}%
}x^{i})=V(s,\xi).
\]

For the general case, note that $V$ is continuous in $x$ and $Y_{s}%
^{s,\xi,v;Q}$ is continuous in $\xi$. We can choose a sequence of simple
random variables $\{\xi^{i}\}$ $(i=1,2,\ldots)$ which converges to $\xi$.
Using similar techniques in Lemma \ref{Lem-basic est-0} and
\ref{Lem-basic est-1}, we have%

\[
E_{P}\mid Y_{s}^{s,\xi,v;Q}-Y_{s}^{s,\xi^{i},v;Q}\mid^{2}\rightarrow0,\text{
\ \ }E_{P}\mid V(s,\xi)-V(s,\xi^{i})\mid^{2}\rightarrow0.
\]
Then (\ref{dpp ine-1}) holds.

\textbf{(ii)} For $\xi\in L_{G}^{2}(\Omega_{s})$, we can construct a random
variable
\[
\eta=\sum\limits_{i=1}^{\infty}1_{A_{i}}x^{i}\in\mathbb{L}^{2}(\Omega_{s}),
\]
such that
\[
\mid\eta-\xi\mid\leq\frac{\varepsilon}{3C}%
\]
where $C:=\max\{C_{0},C_{2}\}.$

By Lemma \ref{Lem-basic est-0} and \ref{Lem-basic est-1}, for any $v(\cdot
)\in\mathcal{U}[s,T],$%

\[
\mid Y_{s}^{s,\eta,v;Q}-Y_{s}^{s,\xi,v;Q}\mid\leq\frac{\varepsilon}{3},\;\mid
V(s,\eta)-V(s,\xi)\mid\leq\frac{\varepsilon}{3}.
\]
For every $x^{i},$ we can choose an\ admissible control $v^{i}(\cdot
)\in\mathcal{U}[s,T]$ and $Q^{i}$ such that
\[
V(s,x^{i})\leq Y_{s}^{s,x^{i},v^{i};Q^{i}}+\frac{\varepsilon}{3}%
,\ \ \ P-a.s..
\]
Denote$\ $%
\[%
\begin{array}
[c]{l}%
v(\cdot):=\sum\limits_{i=1}^{\infty}1_{A_{i}}v^{i}(\cdot)\in\mathbb{U}[s,T],\\
Q^{\prime}(A):=\sum\limits_{i=1}^{\infty}Q^{i}(A\cap A_{i}),\text{
\ \ }\forall A\in\mathcal{B}(\Omega_{T}).
\end{array}
\]
We have%

\[%
\begin{array}
[c]{cl}%
Y_{s}^{s,\xi,v;Q^{\prime}} & \geq-\mid Y_{s}^{s,\eta,v;Q^{\prime}}%
-Y_{s}^{s,\xi,v;Q^{\prime}}\mid+Y_{s}^{s,\eta,v;Q^{\prime}}\\
& \geq-\frac{\varepsilon}{3}+\sum\limits_{i=1}^{\infty}1_{A_{i}}Y_{s}%
^{s,x^{i},v^{i};Q^{i}}\\
& \geq-\frac{\varepsilon}{3}+\sum\limits_{i=1}^{\infty}1_{A_{i}}%
(V(s,x^{i})-\frac{\varepsilon}{3})\\
& =-\frac{2\varepsilon}{3}+\sum\limits_{i=1}^{\infty}1_{A_{i}}V(s,x^{i})\\
& =-\frac{2\varepsilon}{3}+V(s,\eta)\\
& \geq-\varepsilon+V(s,\xi),\ \ \ P-a.s..
\end{array}
\]
\bigskip

\textbf{(iii)} By (\ref{dpp ine-1}) and (\ref{dpp ine-2}), it is easy to prove
(\ref{dpp ine-3}).

The proof is completed. $\ \ \ \ \ \Box$

Define the (backward) semigroup%

\[
G_{t,s}^{t,x,u;P}[\eta]=Y_{t}^{t,x,u;P},
\]
where $\eta\in L_{G}^{2+\varepsilon}(\Omega_{s})$ and $(X_{r}^{t,x,u;P}%
,Y_{r}^{t,x,u;P},Z_{r}^{t,x,u;P})_{t\leq r\leq s}\ $is\ the\ solution of the
following forward-backward system:
\begin{align*}
dX_{s}^{t,x,u;P}  &  =b(s,X_{s}^{t,x,u;P},u_{s})ds+h_{ij}(s,X_{s}%
^{t,x,u;P},u_{s})d\langle B^{i},B^{j}\rangle_{s}+\sigma(s,X_{s}^{t,x,u;P}%
,u_{s})dB_{s},\\
X_{t}^{t,x,u;P}  &  =x
\end{align*}

and%

\[%
\begin{array}
[c]{rl}%
-dY_{r}^{t,x,u;P}= & f(X_{r}^{t,x,u;P},Y_{r}^{t,x,u;P},Z_{r}^{t,x,u;P}%
,u_{r})dr-Z_{r}^{t,x,u;P}dB_{r}\\
& +g_{ij}(X_{r}^{t,x,u;P},Y_{r}^{t,x,u;P},Z_{r}^{t,x,u;P},u_{r})d\langle
B^{i},B^{j}\rangle_{r},\\
Y_{s}^{t,x,u;P}= & \eta,\text{ \ \ }r\in\lbrack t,s],\text{ \ \ }P-a.s..
\end{array}
\]

It is obvious that for $Q\in\mathcal{P}_{M}(t,P)$\
\[
G_{t,T}^{t,x,u;Q}[\Phi(X_{T}^{t,x,u;Q})]=G_{t,s}^{t,x,u;Q}[Y_{s}^{t,x,u;Q}].
\]

Now we give the proof of Theorem \ref{Thm-dpp}:

\noindent\textbf{Proof{. }}By Theorem \ref{Compare},\textbf{{ }}for each fixed
$P\in\mathcal{P}_{M}$, we have
\[%
\begin{array}
[c]{cl}%
V(t,x) & =\underset{u(\cdot)\in\mathcal{U}^{t}[t,T]}{\text{ess}\sup}\text{
}\underset{Q\in\mathcal{P}_{M}(t,P)}{\text{ess}\sup}^{P}\text{ }%
G_{t,T}^{t,x,u;Q}[\Phi(X_{T}^{t,x,u;Q})]\\
& =\underset{u(\cdot)\in\mathcal{U}^{t}[t,s]}{\text{ess}\sup}\text{
}\underset{Q\in\mathcal{P}_{M}(t,P)}{\text{ess}\sup}^{P}\text{ }%
G_{t,s}^{t,x,u;Q}[Y_{s}^{t,x,u;Q}]\\
& =\underset{u(\cdot)\in\mathcal{U}[t,s]}{\text{ess}\sup}\text{ }%
\underset{Q\in\mathcal{P}_{M}(t,P)}{\text{ess}\sup}^{P}G_{t,s}^{t,x,u;Q}%
[Y_{s}^{s,X_{s}^{t,x,u},u;Q}],\text{ \ \ }P-a.s..
\end{array}
\]

By Lemma \ref{Lem-dpp ine} and the comparison theorem of BSDE, we have%

\[
V(t,x)\leq\underset{u(\cdot)\in\mathcal{U}[t,s]}{\text{ess}\sup}\text{
}\underset{Q\in\mathcal{P}_{M}(t,P)}{\text{ess}\sup}G_{t,s}^{t,x,u;Q}%
[V(s,X_{s}^{t,x,u})],\text{ \ \ }P-a.s..
\]

On the other hand, for each fixed $u(\cdot)\in\mathcal{U}[t,s]$,
$\forall\varepsilon>0,$ by Lemma \ref{Lem-dpp ine}, there exist $\bar{u}%
(\cdot)\in\mathbb{U}[s,T]$ and $\tilde{Q}\in\mathcal{P}_{M}(s,Q)$ such that%
\[
V(s,X_{s}^{t,x,u})\leq Y_{s}^{s,X_{s}^{t,x,u},\bar{u};\tilde{Q}}%
+\varepsilon,\text{ \ \ }Y_{t}^{t,x,\tilde{u};\tilde{Q}}\leq V(t,x),
\]
where
\[%
\begin{array}
[c]{l}%
\tilde{u}_{s}=1_{\{t\leq r\leq s\}}u_{r}+1_{\{s<r\leq T\}}\bar{u}_{r}.
\end{array}
\]
By the above inequality and the comparison theorem, we have%
\[%
\begin{array}
[c]{l}%
Y_{s}^{s,X_{s}^{t,x,u},\tilde{u};\tilde{Q}}\geq V(s,X_{s}^{t,x,u}%
)-\varepsilon,\\
V(t,x)\geq G_{t,s}^{t,x,u;Q}[Y_{s}^{s,X_{s}^{t,x,u},\tilde{u};\tilde{Q}}]\geq
G_{t,s}^{t,x,u;Q}[V(s,X_{s}^{t,x,u})-\varepsilon].
\end{array}
\]

By Lemma \ref{Lem-basic est-0}, there exists a constant $C_{0}$ such that%

\[
V(t,x)\geq G_{t,s}^{t,x,u;Q}[V(s,X_{s}^{t,x,u})]-C_{0}\varepsilon.
\]
\bigskip

From this we get%

\[
V(t,x)\geq\underset{u(\cdot)\in\mathcal{U}[t,s]}{\text{ess}\sup}\text{
}\underset{Q\in\mathcal{P}_{M}(t,P)}{\text{ess}\sup}^{P}\text{ }%
G_{t,s}^{t,x,u;Q}[V(s,X_{s}^{t,x,u})]-C_{0}\varepsilon.
\]

Thus by letting $\varepsilon\downarrow0$, we obtain%
\[
V(t,x)=\underset{u(\cdot)\in\mathcal{U}[t,s]}{\text{ess}\sup}\text{
}\underset{Q\in\mathcal{P}_{M}(t,P)}{\text{ess}\sup}^{P}\text{ }%
G_{t,s}^{t,x,u;Q}[V(s,X_{s}^{t,x,u;Q})],\text{ \ \ }P-a.s..
\]
Similar to the proof of Theorem \ref{Compare}, we can get%
\[
V(t,x)=\underset{u(\cdot)\in\mathcal{U}^{t}[t,s]}{\text{ess}\sup}\text{
}\underset{Q\in\mathcal{P}_{M}(t,P)}{\text{ess}\sup}^{P}\text{ }%
G_{t,s}^{t,x,u;Q}[V(s,X_{s}^{t,x,u;Q})],\text{ \ \ }P-a.s..
\]

Note that%
\[%
\begin{array}
[c]{rl}
& \underset{Q\in\mathcal{P}_{M}(t,P)}{\text{ess}\sup}^{P}G_{t,s}%
^{t,x,u;Q}[V(s,X_{s}^{t,x,u})]\\
= & \mathbb{G}_{t,s}^{t,x,u}[V(s,X_{s}^{t,x,u})]\text{ \ \ }P-a.s.
\end{array}
\]

We have%
\[%
\begin{array}
[c]{rl}%
V(t,x)= & \underset{u(\cdot)\in\mathcal{U}[t,s]}{\text{ess}\sup}%
\mathbb{G}_{t,s}^{t,x,u}[V(s,X_{s}^{t,x,u})]\\
= & \underset{u(\cdot)\in\mathcal{U}^{t}[t,s]}{\sup}\mathbb{G}_{t,s}%
^{t,x,u}[V(s,X_{s}^{t,x,u})].
\end{array}
\]
This completes the proof. $\ \ \ \Box$

The following lemma show the continuity of $V$ about $t.$

\begin{lemma}
\label{Lem-dpp ine 2}The value function $V$ is $\frac{1}{2}$ H\"{o}lder
continuous in $t$.
\end{lemma}

\textbf{Proof. }Set $(t,x)\in\mathbb{R}^{n}\times\lbrack0,T]$ and $\delta>0.$
By dynamic programming principle, $\forall$ $\varepsilon>0,$ there exist
$u(\cdot)\in$ $\mathcal{U}^{t}$ such that

\begin{equation}%
\begin{array}
[c]{c}%
\mathbb{G}_{t,t+\delta}^{t,x,u}[V(t+\delta,X_{t+\delta}^{t,x,u})]+\varepsilon
\geq V(t,x)\geq\mathbb{G}_{t,t+\delta}^{t,x,u}[V(t+\delta,X_{t+\delta}%
^{t,x,u})].
\end{array}
\label{ct-1}%
\end{equation}

We first show that there exists $C>0$ such that $V(t+\delta,x)-V(t,x)\leq
C\delta^{\frac{1}{2}}.$ Similarly, we can prove $V(t+\delta,x)-V(t,x)\geq
-C\delta^{\frac{1}{2}}.$

By equation (\ref{ct-1}), we have%

\begin{equation}%
\begin{array}
[c]{c}%
V(t+\delta,x)-V(t,x)\leq I_{\delta}^{1}+I_{\delta}^{2},
\end{array}
\label{ct-2}%
\end{equation}

where
\[%
\begin{array}
[c]{cl}%
I_{\delta}^{1}= & \mathbb{G}_{t,t+\delta}^{t,x,u}[V(t+\delta,x)]-\mathbb{G}%
_{t,t+\delta}^{t,x,u}[V(t+\delta,X_{t+\delta}^{t,x,u})]\\
I_{\delta}^{2}= & V(t+\delta,x)-\mathbb{G}_{t,t+\delta}^{t,x,u}[V(t+\delta,x)]
\end{array}
\]

By Lemma \ref{Lem-basic est-1}, note that $V$ is 1-H\"{o}lder continuous in
$x$. We have%

\[
\left\vert I_{\delta}^{1}\right\vert \leq\lbrack C\mathbb{\hat{E}}\left\vert
V(t+\delta,x)-V(t+\delta,X_{t+\delta}^{t,x,u})\right\vert ]^{\frac{1}{2}}%
\leq\lbrack C\mathbb{\hat{E}}\left\vert X_{t+\delta}^{t,x,u}-x\right\vert
^{2}]^{\frac{1}{2}}.
\]

Then by $\mathbb{\hat{E}}\left\vert X_{t+\delta}^{t,x,u}-x\right\vert ^{2}\leq
C\delta$ ($C$ will change line by line),%

\[
\left\vert I_{\delta}^{1}\right\vert \leq C\delta^{\frac{1}{2}}.
\]

According to the definition of $\mathbb{G}_{t,t+\delta}^{t,x,u},$ $I_{\delta
}^{2}$ can be rewritten as
\[%
\begin{array}
[c]{cl}%
I_{\delta}^{2}= & V(t+\delta,x)-\mathbb{\hat{E}[}V(t+\delta,x)+\int%
_{t}^{t+\delta}f(s,X_{s}^{t,x,u},Y_{s}^{t,x,u},Z_{s}^{t,x,u},u_{s})ds\\
& +\int_{t}^{t+\delta}g_{ij}(s,X_{s}^{t,x,u},Y_{s}^{t,x,u},Z_{s}^{t,x,u}%
,u_{s})d\langle B^{i},B^{j}\rangle_{s}].
\end{array}
\]

It yields that%

\[%
\begin{array}
[c]{rl}%
\left\vert I_{\delta}^{2}\right\vert \leq & \delta^{\frac{1}{2}}%
\{[\mathbb{\hat{E}}\int_{t}^{t+\delta}\left\vert f(s,X_{s}^{t,x,u}%
,Y_{s}^{t,x,u},Z_{s}^{t,x,u},u_{s})\right\vert ^{2}ds]^{\frac{1}{2}}\\
& +[\mathbb{\hat{E}}\int_{t}^{t+\delta}\left\vert g_{ij}(s,X_{s}^{t,x,u}%
,Y_{s}^{t,x,u},Z_{s}^{t,x,u},u_{s})\right\vert ^{2}d\langle B^{i},B^{j}%
\rangle_{s}]^{\frac{1}{2}}\}\\
\leq & C\delta^{\frac{1}{2}}.
\end{array}
\]

Thus, we have
\[
V(t+\delta,x)-V(t,x)\leq C\delta^{\frac{1}{2}}.
\]

This completes the proof. $\ \ \ \Box$

\section{The viscosity solution of HJB equation}

The following theorem gives the relationship between the value function $V$
and the second-order partial differential equation (\ref{hjb}).

\begin{theorem}
\label{viscosity} Let Assumptions \ref{assu-1} and \ref{assu-2} hold. $V$ is
the value function defined by (\ref{value function}). Then $V$ is a viscosity
solution of the following second-order partial differential equation:
\begin{align}
&  \partial_{t}V(t,x)+\sup_{u\in U}H(t,x,V,\partial_{x}V,\partial_{xx}%
^{2}V,u)=0,\label{hjb}\\
&  V(T,x)=\Phi(x),\quad\ \ x\in\mathbb{R}^{n},\nonumber
\end{align}
where%
\[%
\begin{array}
[c]{cl}%
H(t,x,v,p,A,u)= & G(F(t,x,v,p,A,u))+\langle p,b(t,x,u)\rangle+f(t,x,v,\sigma
(t,x,u)p,u),\\
F_{ij}(t,x,v,p,A,u)= & \langle A\sigma_{i}(t,x,u),\sigma_{j}(t,x,u)\rangle
+2\langle p,h_{ij}(t,x,u)\rangle+2g_{ij}(t,x,v,\sigma(t,x,u)p,u),
\end{array}
\]
$(t,x,v,p,A,u)\in\lbrack0,T]\times\mathbb{R}^{n}\times\mathbb{R}%
\times\mathbb{R}^{d}\times\mathbb{S}_{n}\times U$, $\sigma_{i}$ is the $i$-th
column of $\sigma$, $G$ is defined by equation (\ref{Gequation}).
\end{theorem}

For simplicity, we only consider the case $h_{ij}=g_{ij}=0$.

Suppose $\varphi\in C_{b,Lip}^{2,3}([t,T]\times\mathbb{R}^{n})$. Define%
\begin{equation}%
\begin{array}
[c]{l}%
F_{1}(s,x,y,z,u)=\partial_{s}\varphi(s,x)+\langle b(s,x,u),\partial_{x}%
\varphi(s,x)\rangle+f(s,x,y+\varphi(s,x),z+\partial_{x}\varphi(s,x)\sigma
(s,x,u),u),\\
F_{2}^{ij}(s,x,u)=\frac{1}{2}\langle\partial_{xx}^{2}\varphi(s,x)\sigma
_{i}(s,x,u),\sigma_{j}(s,x,u)\rangle.
\end{array}
\label{EQS-1}%
\end{equation}

Consider the following G-BSDEs: $\forall s\in\lbrack t,t+\delta],$%

\begin{equation}%
\begin{array}
[c]{rl}%
Y_{s}^{1,u}= & \int_{s}^{t+\delta}F_{1}(r,X_{r}^{t,x,u},Y_{r}^{1,u}%
,Z_{r}^{1,u},u_{r})dr+\int_{s}^{t+\delta}F_{2}^{ij}(r,X_{r}^{t,x,u}%
,u_{r})d\langle B^{i},B^{j}\rangle_{r}\\
& -\int_{s}^{t+\delta}Z_{r}^{1,u}dB_{r}-(K_{t+\delta}^{1,u}-K_{s}^{1,u}),\\
Y_{t+\delta}^{1,u}= & 0,
\end{array}
\label{auxilary-1}%
\end{equation}
and%
\begin{equation}
Y_{s}^{u}=\varphi(t+\delta,X_{t+\delta}^{t,x,u})+\int_{s}^{t+\delta}%
f(r,X_{r}^{t,x,u},Y_{r}^{u},Z_{r}^{u},u_{r})dr-\int_{s}^{t+\delta}Z_{r}%
^{u}dB_{r}-(K_{t+\delta}^{u}-K_{s}^{u}). \label{equa-1}%
\end{equation}

\begin{lemma}
$\forall s\in\lbrack t,t+\delta],$ we have%

\begin{equation}
Y_{s}^{1,u}=Y_{s}^{u}-\varphi(s,X_{s}^{t,x,u}). \label{auxilary-1-sol}%
\end{equation}

\end{lemma}

\noindent\textbf{Proof. }Applying It\^{o}'s formula to $\varphi(s,X_{s}%
^{t,x,u})$, we have%

\[
d(Y_{s}^{u}-\varphi(s,X_{s}^{t,x,u}))=dY_{s}^{1,u}.
\]

Since $Y_{t+\delta}^{u}-\varphi(t+\delta,X_{t+\delta}^{t,x,u})=Y_{t+\delta
}^{1,u}=0,$ we obtain%

\[
Y_{s}^{1,u}=Y_{s}^{u}-\varphi(s,X_{s}^{t,x,u}),\text{ \ \ }\forall s\in\lbrack
t,t+\delta].
\]

The proof is completed. $\ \ \ \ \ \Box$

Consider the G-BSDE: $\forall s\in\lbrack t,t+\delta],$%

\begin{equation}%
\begin{array}
[c]{rl}%
Y_{s}^{2,u}= & \int_{s}^{t+\delta}F_{1}(r,x,Y_{r}^{2,u},Z_{r}^{2,u}%
,u_{r})dr+\int_{s}^{t+\delta}F_{2}^{ij}(r,x,u_{r})d\langle B^{i},B^{j}%
\rangle_{r}-\int_{s}^{t+\delta}Z_{r}^{2,u}dB_{r}-(K_{t+\delta}^{2,u}%
-K_{s}^{2,u}).
\end{array}
\label{auxilary-2}%
\end{equation}

We have the following estimation.

\begin{lemma}
We have
\begin{equation}
\mid\underset{u(\cdot)\in\mathcal{U}[t,T]}{\text{ess}\sup}Y_{t}^{1,u}%
-\underset{u(\cdot)\in\mathcal{U}[t,T]}{\text{ess}\sup}Y_{t}^{2,u}\mid\leq
C\delta^{3/2}, \label{auxilary-ine}%
\end{equation}
where $C$ is a positive constant independent of $u(\cdot)$.
\end{lemma}

\noindent\textbf{Proof.}\ By Proposition 3.9 in \cite{HJPS1}, we have for any
fixed $u(\cdot)\in\mathcal{U}[t,T]$ and $p>2$%
\begin{align*}
|Y_{t}^{1,u}-Y_{t}^{2,u}|^{2}  &  \leq\mathbb{\hat{E}}[\underset{s\in\lbrack
t,t+\delta]}{\sup}|Y_{s}^{1,u}-Y_{s}^{2,u}|^{2}]\\
&  \leq C\{(\mathbb{\hat{E}}[\sup_{s\in\lbrack t,t+\delta]}\mathbb{\hat{E}%
}_{s}[(\int_{t}^{t+\delta}\hat{F}_{r}dr)^{p}]])^{2/p}+\mathbb{\hat{E}}%
[\sup_{s\in\lbrack t,t+\delta]}\mathbb{\hat{E}}_{s}[(\int_{t}^{t+\delta}%
\hat{F}_{r}dr)^{p}]]\},
\end{align*}
where $\hat{F}_{r}=|F_{1}(r,X_{r}^{t,x,u},Y_{r}^{2,u},Z_{r}^{2,u},u_{r}%
)-F_{1}(r,x,Y_{r}^{2,u},Z_{r}^{2,u},u_{r})|+\sum_{i,j=1}^{d}|F_{2}%
^{ij}(r,X_{r}^{t,x,u},Y_{r}^{2,u},Z_{r}^{2,u},u_{r})-F_{2}^{ij}(r,x,Y_{r}%
^{2,u},Z_{r}^{2,u},u_{r})|$. It is easy to verify that
\[
\hat{F}_{r}\leq C_{1}(|X_{r}^{t,x,u}-x|+|X_{r}^{t,x,u}-x|^{2}),
\]
where $C_{1}$ is independent of $u(\cdot)$. By standard estimates of G-SDE, we
can obtain that for any $p^{\prime}\geq2$%
\[
\mathbb{\hat{E}}[\underset{r\in\lbrack t,t+\delta]}{\sup}\mid X_{r}%
^{t,x,u}-x\mid^{p^{\prime}}]\leq C_{2}(1+|x|^{p^{\prime}})\delta^{p^{\prime
}/2},
\]
where $C_{2}$ is independent of $u(\cdot)$. Then by Theorem 2.13 in
\cite{HJPS1} we can deduce that $|Y_{t}^{1,u}-Y_{t}^{2,u}|\leq C\delta^{3/2}$,
where $C$ is independent of $u(\cdot)$. Thus%
\[
\mid\underset{u(\cdot)\in\mathcal{U}[t,T]}{\text{ess}\sup}Y_{t}^{1,u}%
-\underset{u(\cdot)\in\mathcal{U}[t,T]}{\text{ess}\sup}Y_{t}^{2,u}\mid
\leq\underset{u(\cdot)\in\mathcal{U}[t,T]}{\text{ess}\sup}|Y_{t}^{1,u}%
-Y_{t}^{2,u}|\leq C\delta^{3/2}.
\]

This completes\ the proof. $\ \ \ \ \ \Box$

Now we compute $\underset{u(\cdot)\in\mathcal{U}[t,t+\delta]}{\text{ess}\sup}$
$Y_{t}^{2,u}.$

\begin{lemma}
\label{Lem-auxilary-3}We have
\[
\underset{u(\cdot)\in\mathcal{U}[t,t+\delta]}{\text{ess}\sup}\text{ }%
Y_{t}^{2,u}=Y_{t}^{0},
\]
where $Y^{0}$ is the solution of the following ordinary differential equation%
\begin{equation}
-dY_{s}^{0}=F_{0}(s,x,Y_{s}^{0},0)ds,\text{ \ \ }Y_{t+\delta}^{0}=0,\text{
}s\in\lbrack t,t+\delta] \label{auxilary-3}%
\end{equation}
and%
\[%
\begin{array}
[c]{rl}%
F_{0}(s,x,y,z):= & \underset{u\in U}{\sup}[F_{1}(s,x,y,z,u)+2G(F_{2}(s,x,u))].
\end{array}
\]

\end{lemma}

\noindent\textbf{Proof. }By Theorem 3.7 of \cite{HJPS}, we have%

\[
Y_{s}^{2,u}\leq Y_{s}^{0},\text{ \ \ }s\in\lbrack t,t+\delta],
\]
where $(Y^{0},Z^{0},K^{0})$ is the solution of the following G-BSDE:%
\[%
\begin{array}
[c]{cl}%
Y_{s}^{0}= & \int_{s}^{t+\delta}F_{0}(r,x,Y_{r}^{0},Z_{r}^{0})dr-\int%
_{s}^{t+\delta}Z_{r}^{0}dB_{r}-(K_{t+\delta}^{0}-K_{s}^{0})\text{ \ \ }%
s\in\lbrack t,t+\delta].
\end{array}
\]
Since $F_{1}$ and $G(F_{2})$ are deterministic functions, we obtain that
$Z_{s}^{0}=0,K_{s}^{0}=0$ and $Y_{s}^{0}$ is the solution of equation
(\ref{auxilary-3}).

We denote the class of all deterministic controls in $\mathcal{U}[t,t+\delta]$
by $\mathcal{U}_{1}$. Then, for every $u(\cdot)\in\mathcal{U}_{1},$ $Y^{2,u}$
is the solution of the following ordinary differential equation:%

\[%
\begin{array}
[c]{l}%
-dY_{s}^{2,u}=[F_{1}(s,x,Y_{s}^{2,u},0,u_{s})+2G(F_{2}(s,x,u_{s}))]ds,\text{
\ \ }s\in\lbrack t,t+\delta],\\
Y_{t+\delta}^{2,u}=0.
\end{array}
\]
It is easy to check that%

\[
Y_{t}^{0}=\underset{u(\cdot)\in\mathcal{U}[t,t+\delta]}{\text{ess}\sup}\text{
}Y_{t}^{2,u}.
\]

This completes the proof.$\ \ \ \ \ \ \Box$

Finally we give the proof of Theorem \ref{viscosity}.

\noindent\textbf{Proof: }By Lemma \ref{Lem-basic est-1}, \ref{Lem-dpp ine 2},
$V$ is a continuous functions on $[0,T]\times\mathbb{R}^{n}$. We first prove
that $V$ is the subsolution of (\ref{hjb}).

Given $t\leq T$ and $x\in\mathbb{R}^{n}$, suppose $\varphi\in C_{b,Lip}%
^{2,3}([0,T]\times\mathbb{R}^{n})$ such that $\varphi(t,x)=V(t,x)$ and
$\varphi\geq V$ on $[0,T]\times\mathbb{R}^{n}$. By Theorem\ \ref{Thm-dpp}, we
have%
\[%
\begin{array}
[c]{rl}%
V(t,x)= & \underset{u(\cdot)\in\mathcal{U}[t,t+\delta]}{\text{ess}\sup
}\mathbb{G}_{t,t+\delta}^{t,x,u;Q}[V(t+\delta,X_{t+\delta}^{t,x,u;Q})].
\end{array}
\]

So%
\[
\underset{u(\cdot)\in\mathcal{U}[t,t+\delta]}{\text{ess}\sup}\text{
}\{\mathbb{G}_{t,t+\delta}^{t,x,u}[\varphi(t+\delta,X_{t+\delta}%
^{t,x,u})]-\varphi(t,x)\}\geq0.
\]
By (\ref{auxilary-1-sol}), we have
\[
\underset{u(\cdot)\in\mathcal{U}[t,t+\delta]}{\text{ess}\sup}Y_{t}^{1,u}%
\geq0.
\]
By (\ref{auxilary-ine}) and Lemma \ref{Lem-auxilary-3}, we get
\[
\underset{u(\cdot)\in\mathcal{U}[t,t+\delta]}{\text{ess}\sup}Y_{t}^{2,u}%
\geq-C\delta^{3/2}%
\]
and%

\[
Y_{t}^{0}\geq-C\delta^{3/2}.
\]
Thus,%
\[
-C\delta^{1/2}\leq\delta^{-1}Y_{t}^{0}=\delta^{-1}\int_{t}^{t+\delta}%
F_{0}(r,x,Y_{r}^{0},0)dr.
\]
Letting $\delta\rightarrow0$, we get $F_{0}(t,x,0,0)=\sup_{u\in U}%
(F_{1}(t,x,y,z,u)+G(F_{2}(t,x,u)))\geq0$, which implies that $V$ is a
subsolution of (\ref{hjb}). Using the same method, we can\ prove $V$ is the
supersolution of (\ref{hjb}).\quad

This completes the proof.$\ \ \ \ \ \ \Box$

\renewcommand{\refname}{\large References}

\bigskip

\bigskip

\textbf{Attachment}

\bigskip

\bigskip

\textbf{Acknowledgments}

The authors would like to thank S. Peng for many helpful discussions.

\end{document}